\documentclass[11pt]{amsart}
\usepackage[english]{babel}
\usepackage{amsmath,amssymb,amsthm}

\usepackage{graphicx}

\usepackage[pdftex]{color}
\usepackage[bookmarks=true,hyperindex,pdftex,colorlinks]{hyperref}

\theoremstyle{plain}
\newtheorem{theorem}{Theorem}[section]
\newtheorem{corollary}[theorem]{Corollary}

\newtheorem{proposition}[theorem]{Proposition}
\newtheorem{lemma}[theorem]{Lemma}

\theoremstyle{definition}

\newtheorem{remark}[theorem]{Remark}

\newtheorem{definition}[theorem]{Definition}

 \DeclareMathOperator{\e}{e}
 
 \DeclareMathOperator{\dist}{dist\,}

\newcommand{\C}{\mathbb{C}}
\newcommand{\R}{\mathbb{R}}
\newcommand{\N}{\mathbb{N}}

\renewcommand{\leq}{\leqslant}
\renewcommand{\geq}{\geqslant}

\begin{document}

\begin{abstract}
This paper deals with the \emph{Bishop-Phelps-Bollob\'as property} (\emph{BPBp} for short) on bounded closed convex subsets of a Banach space $X$, not just on its closed unit ball $B_X$. We firstly prove that the \emph{BPBp} holds for bounded linear functionals on arbitrary bounded closed convex subsets of a real Banach space.  We show that for all finite dimensional Banach spaces $X$ and $Y$ the pair $(X,Y)$ has the \emph{BPBp} on every bounded closed convex subset $D$ of $X$, and also that for a Banach space $Y$ with property $(\beta )$ the pair $(X,Y)$ has the \emph{BPBp} on every bounded closed absolutely convex subset $D$ of an arbitrary Banach space $X$. For a bounded closed absorbing convex subset $D$ of $X$ with positive modulus convexity we get that the pair $(X,Y)$ has the \emph{BPBp} on $D$ for every Banach space $Y$. We further obtain that for an Asplund space $X$ and for a locally compact Hausdorff $L$, the pair $(X, C_0(L))$ has the \emph{BPBp} on every bounded closed absolutely convex subset $D$ of $X$. Finally we study the stability of the \emph{BPBp} on a bounded closed convex set for the $\ell_1$-sum or $\ell_{\infty}$-sum of a family of Banach spaces.
\end{abstract}

\title{The Bishop-Phelps-Bollob\'{a}s theorem on bounded closed convex sets}

\author[Cho]{Dong Hoon Cho}
\address[Cho]{Department of Mathematics, POSTECH, Pohang (790-784), Republic of Korea}
\email{\texttt{meimi200@postech.ac.kr}}

\author[Choi]{Yun Sung Choi}
\address[Choi]{Department of Mathematics, POSTECH, Pohang (790-784), Republic of Korea}
\email{\texttt{mathchoi@postech.ac.kr}}

\thanks{The first author is a corresponding author.}
\subjclass[2000]{}

\date{September 9th, 2014}

\maketitle

\section{Introduction}

 A remarkable result so called the \emph{Bishop-Phelps theorem} \cite{BishopPhelps} came out in 1961, which states that for every Banach space $X$, every linear functional on $X$ can be approximated by norm attaining ones. In fact, they showed a more general results: Let $D$ be a closed bounded convex subset of a real Banach space
 $X$. Then the set of support functionals of $D$ is a norm dense subset of its dual space $X^*$. In other words, the set of all elements of $X^*$ that attain their suprema on $D$ is a norm dense subset of $X^*$.
 However, Lomonosov \cite{Lomono} showed in 2000 that this statement cannot be extended to general complex spaces by constructing a closed bounded convex set with no support points. From now on, we assume that $X$ and $Y$ are real Banach spaces without any other comment.

 After a while, J. Lindenstrauss \cite{Lindens} studied in 1963 the denseness of norm attaining linear operators between Banach spaces, which has been a classical research topic in functional analysis since then. In particular, Bourgain \cite{Bou} obtained in 1976 such a surprising results that a Banach space $X$ has the \emph{Bishop-Phelps property} if and only if it has the Radon-Nikodym property(\emph{RNP} for short). We recall that a Banach space $X$ is said to have \emph{Bishop-Phelps property} if for every bounded closed and absolutely convex subset $D$ of $X$ and for every Banach space $Y$, the subset of $\mathcal{L}(X,Y)$ attaining their suprema in norm on $D$ is dense in the space $\mathcal{L}(X,Y)$, where $\mathcal{L}(X,Y)$ is the Banach space of bounded linear operators from $X$ into $Y$.

 In 1977 Stegall \cite{Stegall2} obtained a nonlinear form of Bourgain's result: Let $X$ be a Banach space with RNP, $D$ be a bounded closed convex subset of $X$ and $f:D\rightarrow \mathbb{R}$ be an upper semicontinuous bounded above function. Then for $\epsilon>0$, there exists $x^{*}\in X^{*}$ such that $\|x^{*}\|<\epsilon$ and $f+x^{*}$, $f+|x^{*}|$ strongly expose $D$. Applying this result to a vector-valued case, he showed the following. Let $X$ be a Banach space with \emph{RNP}, $D$ be a bounded closed convex subset of $X$, and $Y$ be a Banach space. Suppose that $\varphi : D\to Y$ is a uniformly bounded function such that the function $x\to \|\varphi(x)\|$ is upper semicontinuous. Then, for $\delta>0$, there exist $T:X\to Y$ a bounded linear operator of rank one, $\|T\|< \delta$ such that $\varphi +T$ attains its supremum in norm on $D$ and does so at most two points.

 We refer to \cite{Acosta-RACSAM} surveying most of recent results on the denseness of norm attaining linear or nonlinear mappings such as multilinear mappings, polynomials or holomorphic mappings.

 On the other hand Bollob\'{a}s \cite{Bollobas} sharpened in 1970 the Bishop-Phelps theorem by dealing simultaneously with norm attaining linear functionals and their norming points, which is stated as follows. We denote by $B_X$ and $S_X$ the closed unit ball and sphere of $X$, respectively.

\begin{theorem}\cite{Bollobas}
Let $X$ be a Banach space and $0<\epsilon <1$. Given $x \in S_X$ and $x^{*}\in S_{X^{*}}$ with $|1-x^{*}(x)|<\frac{\epsilon^{2}}{2}$,
there are elements $y \in S_X$ and $y^{*}\in S_{X^{*}}$ such that
$$
y^{*}(y)=1,~~~\|x-y\|<\epsilon,~~~and~~~\|y^{*}-x^{*}\|<\epsilon+\epsilon^2.
$$
\end{theorem}

He also showed that this theorem is best possible in the following sense. For any $0<\epsilon <1$ there exist a Banach space $X$, point $x\in S_X$ and functional $f\in S_{X^*}$ such that $f(x)= 1- (\epsilon^2/2)$, but if $y\in S_X$, $g\in S_{X^*}$ and $g(y)=1$, then either $\|f-g\|\geq \epsilon$ or $\|x-y\|\geq \epsilon$.

Since this theorem of Bollob\'{a}s is stated explicitly, we have referred it more often than the theorem of Br$\o$nsted and Rockafellar \cite{BR}, a more general and earlier result than Bollob\'{a}s. Using the concept of the subdifferential of a convex function it is written as follows: Suppose that $f$ is a convex proper lower semicontinuous function on a Banach space $X$. Then given any point $x_0 \in dom(f)$, $\epsilon>0$, $\lambda>0$ and any $x_0^*\in \partial_{\epsilon}f(x_0)$, there exist $x\in dom(f)$ and $x^*\in X^*$ such that $$x^*\in \partial(f),~ \|x-x_0\|\leq \frac{\epsilon}{\lambda},~\mbox{and},~\|x^*-x_0^*\|\leq \lambda.$$ In particular, the domain of $\partial f$ is dense in $dom(f)$.

Acosta et al. \cite{AAGM2} introduced in 2008 the following definition to study this property for linear operators between Banach spaces.

\begin{definition}\cite{AAGM2}
A pair of Banach spaces $(X,Y)$ is said to have the \emph{Bishop-Phelps-Bollob\'{a}s property} (\emph{BPBp} for short) if for every $\epsilon >0$ there are $0<\eta(\epsilon)<1$ and
$\beta(\epsilon)>0$ with $\lim_{\epsilon\rightarrow 0}\beta(\epsilon)=0$ such that for all $T\in S_{\mathcal{L}(X,Y)}$ and $x_0 \in S_X$ satisfying
$\|T(x_0)\|>1-\eta(\epsilon)$, there exist a point $u\in S_X$ and an operator $S\in S_{\mathcal{L}(X,Y)}$ that satisfy the following conditions:
$$
\|Su_0\|=1,~~~\|u_0 -x_0\|<\beta(\epsilon),~~~and~~~\|T-S\|<\epsilon.
$$
\end{definition}

Since they characterized  in \cite{AAGM2} the Banach
space $Y$ for which the \emph{BPBp} holds for operators from $\ell_1$  into $Y$, lots of interest has been caused in this property (for instance see \cite{ACKLM, ADM, ACK, ABGM, ACGM-Adv, DLM, Kim-c_0, KimLee}).

 We note that the \emph{BPBp} is not so closely related with \emph{RNP} as the \emph{Bishop-Phelps-Bollob\'{a}s property}. For example, $\ell_1$ has \emph{RNP}, but there exists a Banach space $Y$ such that the pair $(\ell_1,Y)$ does not have the \emph{BPBp} (\cite{AAGM2}). On the other hand, the pair $(L_1[0,1], L_{\infty}[0,1])$ has the \emph{BPBp} (\cite{ACGM-Adv}), but $L_1[0,1]$ does not have \emph{RNP}.

So far, the \emph{BPBp} has been studied on the closed unit ball $B_X$, but in this paper we deal with this property on bounded closed convex subsets $D$ of a Banach space $X$, not just on $B_X$. We introduce the following more general definition. Let $$\|T\|_D =\sup\{ \|Tx\|:x\in D\}$$ for $T\in \mathcal{L}(X,Y)$.

\begin{definition}
Let $X$ and $Y$ be Banach spaces. Let $D$ be a bounded closed convex subset of $X$.
We say that $(X,Y)$ has the \emph{Bishop-Phelps-Bollob\'{a}s property} on $D$ (\emph{BPBp} on $D$ for short) if for every $\epsilon>0$, there is $\eta_{D}(\epsilon)>0$ such that for every $T\in L(X,Y),~ \|T\|_D =1$ and every $x \in D$ satisfying
$$
\|T(x)\|> 1-\eta_{D}(\epsilon),
$$
there exist $S\in L(X,Y)$ and $z\in D$ such that
$$
\|S(z)\|=1=\|S\|_D,~~~\|x-z\|<\epsilon~~~and~~~\|T-S\|<\epsilon.
$$
Similarly we say that $(X,Y)$ has the \emph{Bishop-Phelps property} on $D$ (\emph{BPp} on $D$ for short) if for every $\epsilon>0$ and for every $T\in \mathcal{L}(X,Y)$ with $\|T\|_D =1$, then there exist $S \in \mathcal{L}(X,Y)$ and $z \in D$ such that
$$
\|S(z)\|=1=\|S\|_D~~~and~~~\|T-S\|<\epsilon.
$$
\end{definition}

In general, we cannot expect the same results in the \emph{BPBp} on a closed bounded convex set $D$ as those on $B_X$. For a uniformly convex space $X$ the pair $(X,Y)$ has the \emph{BPBp} on $B_X$ for every Banach space $Y$ (\cite{ABGM, KimLee}). However, there is a Banach space $Y$ such that $(\ell_2^2, Y)$ fails to have the \emph{BPBp} on $D=B_{\ell_1^2}$, even though $\ell_2^2$ is a uniformly convex space of dimension 2. We can actually show this fact by just considering the bounded operators $T_k$ defined on $\ell_1^2$ in \cite[Example 4.1]{ACKLM} as those on $\ell_2^2$. In fact, for $k\in \mathbb{N}$, consider $Y_{k}=\mathbb{R}^{2}$ with the norm
$$
\|(x,y)\|=\max\left\{|x|,|y|+\frac{1}{k}|x|\right\},
$$
and $\mathcal{Y}=[\bigoplus^{\infty}_{k=1}Y_{k}]_{\ell_{\infty}}$. Define $T_{k}\in \mathcal{L}(\ell^{2}_{2},Y_{k})$ by
$$
T_{k}(e_1)=\left(-1,1-\frac{1}{k}\right)~~~~and~~~~T_{k}(e_2)=\left(1,1-\frac{1}{k}\right).
$$
Since $B_{\ell^{2}_{1}}=D \subset B_{\ell^{2}_{2}}$, we have $\|T_{k}\|_{D}=1$ for all $k\in \mathbb{N}$.
It follows from the same argument as in \cite[Example 4.1]{ACKLM} with Theorem \ref{beta} and Proposition \ref{ell-infty}  that $(\ell_2^2, Y)$ fails to have the \emph{BPBp} on $D=B_{\ell_1^2}$.

In Section $2$, we show that the \emph{BPBp} holds for bounded linear functionals on arbitrary bounded closed convex sets. Using this, we sharpen Stegall's optimization principles \cite{Stegall2} in the sense of the \emph{BPBp}.  In Section $3$, we show that for all finite dimensional Banach spaces $X$ and $Y$ the pair $(X,Y)$ has the \emph{BPBp} on every bounded closed convex subset $D$ of $X$, and also that for a Banach space $Y$ with property $(\beta )$ the pair $(X,Y)$ has the \emph{BPBp} on an every bounded closed absolutely convex subset $D$ of an arbitrary Banach space $X$.  For a bounded closed convex subset $D$ of $X$ with positive modulus convexity we get that the pair $(X,Y)$ has the \emph{BPBp} on $D$ for every Banach space $Y$. We further prove that for an Asplund space $X$ and for a locally compact Hausdorff $L$, the pair $(X, C_0(L))$ has the \emph{BPBp} on every bounded closed absolutely convex subset $D$ of $X$. In Section $4$, we study the stability of the \emph{BPBp} for the $\ell_1$-sum or $\ell_{\infty}$-sum of a family of Banach spaces.

\section{Linear functionals attaining their suprema on bounded closed convex sets}

We begin by recalling Ekeland's variational principle \cite{FHHMZ, eke}, which can be stated as follows:
\begin{theorem}[Ekeland]
Let $f : X\to \R\cup\{\infty\}$ be a proper lower semicontinuous and bounded below function on a Banach space $X$. Then given $\epsilon >0$ and $\delta>0$, there exists $x_1 \in X$ such that
$f(x_1)< f(x)+ \epsilon \|x-x_1\|$ for every $x \in X$ with $x\neq x_1$. Moreover if $f(x_0)< b + \frac{\delta}{2}$, where $b= inf\{f(x)~:~ x\in X\}$, then $x_1$ can be chosen so that $\|x_0-x_1\| < \frac{\delta}{\epsilon}.$
\end{theorem}

The proof of Theorem 7.41 in \cite{FHHMZ} actually gives the following:

\begin{theorem}\label{first}
Let $D$ be a bounded convex closed subset of a Banach space $X$. Given $\epsilon >0$ and $\delta>0$, if $f \in X^{*}$ and $x_0 \in D$ such that
$$
f(x_0)>sup\{ f(x):x \in D\}-\frac{\delta}{2},
$$
then there exist $g \in X^{*}$ and $x_1 \in D$ satisfying
$$
g(x_1)=sup\{g(x):x\in D\},~~\|f-g\| \leq \epsilon ~~and ~~\|x_1-x_0\| \leq \frac{\delta}{\epsilon}.
$$
\end{theorem}

 It is trivial  that Theorem \ref{first} is not true any more for a complex Banach space (\cite{Lomono}).

\begin{corollary}
Let $\epsilon >0$ be given. If $f \in S_{X^{*}}$ and $x_0 \in S_X$ satisfy that
$$
|1-f(x_0)| < \frac{\epsilon^2}{4},
$$
then there exist $g\in S_{X^{*}}$ and $z\in S_X$ such that
$$
g(z)=1,~~\|f-g\|\leq\epsilon ~~and~~\|x_0 -z\| \leq \epsilon.
$$
\end{corollary}

\begin{proof}
Apply Theorem \ref{first} with $\delta= \frac{\epsilon^2}{2}$ and $\epsilon' = \frac{\epsilon}{2}$  and we can choose $\|x^{*}\|\leq \frac{1}{2}\epsilon$ and $z\in S_X$ so that $f+x^{*}$ attains its norm at $z$. Set $g=(f+x^{*})/\|f+x^{*}\|$. Then
\begin{eqnarray*}
\|f-g\|&\leq & \|f-(f+x^{*})\|+\|(f+x^{*})-g \| \leq \frac{1}{2}\epsilon + \|f+x^{*}\|\cdot\left|1-\frac{1}{\|f+x^{*}\|}\right|\\
&\leq & \frac{1}{2}\epsilon + \left| \|f+x^{*}\|-1\right| \leq\frac{1}{2}\epsilon+\frac{1}{2}\epsilon=\epsilon.
\end{eqnarray*}
Also we get $|g(z)|=1$ and $\|x_0-z\|\leq \frac{\epsilon^2}{2}/ \frac{\epsilon}{2} = \epsilon$.
\end{proof}

We can also obtain the following theorem for a bounded linear functional, which is analogous to Stegall's  nonlinear form \cite{Stegall2} of Bourgain's result mentioned in Introduction.

\begin{theorem}\label{pre}
Let $D$ be a bounded closed convex set in a Banach space $X$.
Given $0< \epsilon< 1/4$ and $f \in X^{*}$, there exist $x^{*} \in X^{*}$ and $x_0 \in D$ such that
both $f+x^{*}$ and $f+|x^{*}|$ attain their suprema simultaneously at $x_0$ and $\|x^{*}\|\leq\epsilon$. Moreover $(f+x^{*})(x_0)=(f+|x^{*}|)(x_0)$.
\end{theorem}

\begin{proof}
We may assume $D \subset B_X$ and $\|f\|_D=1$.
By the Bishop-Phelps theorem, there exists $x^{*}\in X^{*}$ such that $f+x^{*}$ attains its supremum at $x_0 \in D$ and $\|x^{*}\|\leq \frac{\epsilon}{2}.$
If $f(x_0)+x^{*}(x_0)\geq f(x)+|x^{*}(x)|$ for every $x\in D$, we are done.
Otherwise, there exists $y \in D$ such that $f(y)+|x^{*}(y)|>f(x_0)+x^{*}(x_0)$. Clearly $x^{*}(y)<0$, and
$$
f(y)-x^{*}(y) > f(x_0)+x^{*}(x_0).
$$
Let $s=\sup_{x\in D}\{f(x)-x^{*}(x)\}$ and $\alpha=s-(f(x_0)+x^{*}(x_0))< (1+ \frac{\epsilon}{2})-(1-\frac{\epsilon}{2})= \epsilon.$
Choose $y_0\in D$ so that $f(y_0)-x^{*}(y_0)>s-\frac{\alpha^{2}\epsilon^{2}}{2}$. By Theorem \ref{first}, there exists $x^{*}_{1}$ such that
$(f-x^{*})+x^{*}_{1}$ attains its supremum at $z_0 \in D$, $\|x^{*}_{1}\|\leq\alpha\epsilon$ and $\|y_0-z_0\|\leq \alpha\epsilon.$
Then,

\begin{eqnarray*}
f(z_0)-x^{*}(z_0)+x^{*}_{1}(z_0) &\geq&
f(y_0)-x^{*}(y_0)-\|x_1^*\|~\|z_0 -y_0\|+ x^{*}_{1}(z_0)
\\[5pt]
&>& s-\frac{3\alpha^{2}\epsilon^{2}}{2}-\alpha\epsilon \\[5pt]
&>& f(x_0)+x^{*}(x_0)+\alpha\epsilon,
\end{eqnarray*}
where the last inequality follows from the definition of $\alpha$ and $0<\alpha < \epsilon <1/4$.
Set $x^{*}_{2}=-x^{*}+x^{*}_{1}$. Clearly, $\|x_2^*\|\leq \epsilon$. We can see that $f + |x_2^*|$ also attains its supremum  $f(z_0) + z_2^*(z_0)$ at $z_0$ on $D$.
Otherwise, there exists $w \in D$ such that
$$
f(w)+|x_2^*|(w) > f(z_0)+ x_2^*(z_0),
$$
which implies that $x_2^*(w) <0$. Therefore, we have

$$ f(w)-x_2^*(w) > f(z_0)+x_2^*(z_0) > f(x_0) + x^*(x_0) + \alpha\epsilon\geq f(w) + x^*(w) +\|x_1^*\|,$$
which implies that $-x^{*}_{1}(w)>\|x^{*}_{1}\|$.
This contradiction shows that
$$
f(z_0)+x^{*}_{2}(z_0) = f(z_0)+ |x^{*}_{2}(z_0)| \geq f(x)+|x^{*}_{2}(x)|
$$
for every $x\in D$.
\end{proof}

By the same argument as in the proof of Theorem \ref{first} we can obtain the following: If $f(x_0)<\inf \{f(x):x:\in D\} + \frac{\delta}{2}$, then there exist $g \in X^{*}$ and $x_1 \in D$ such that
$$
g(x_1)=\inf\{g(x):x\in D\},~~~\|f-g\| \leq \epsilon~~~and~~~\|x_1 -x_0 \|\leq \frac{\delta}{\epsilon}.
$$
 Further, we can show in the following that for a bounded closed convex set $D$ the set $\{f: |f| ~\mbox{attains its supremum on}~ D\}$ is dense in $X^*$.

\begin{theorem} \label{main}
Let $D$ be a bounded closed convex set in a Banach space $X$. Given $f\in X^{*}$ and $\epsilon>0$, there exists $x^{*} \in X^{*}$ such that
$|f+x^{*}|$ attains its supremum on $D$ and $\|x^{*}\|\leq \epsilon$. Moreover, if $D$ is symmetric, and $f(x_0)>\|f\|_{D}-\frac{\delta}{2}$ for some
$x_0\in D$ and $\delta>0$, then $x^{*}$ and $x_1 \in D$ can be chosen so that $\|x^{*}\|\leq\epsilon$,
$\| x_0-x_1\| \leq \frac{\delta}{\epsilon}$, and $|f + x^*|$ attains its supremum at $x_1$ on $D$.
\end{theorem}

\begin{proof} We may assume that $D$ is a bounded closed convex subset of $B_X$. Let $s=\sup_{D}f$. We now consider three cases.

$1^{\circ}$.  Suppose that $ s > |\inf_{D}f|$. Set $\eta = s- |\inf_{D}f| >0$.
By Theorem \ref{pre}, we can choose $\|x^{*}\|\leq \min\{\frac{\eta}{2},\epsilon\}$ so that
both $(f+x^{*})$ and $(f+|x^{*}|)$ attain their suprema at $x_0 \in D$ and
$(f+x^{*})(x_0)=(f+|x^{*}|)(x_0)$. Since $f(x_0)+x^{*}(x_0) \geq f(x)+|x^{*}(x)|$ for every $x\in D$, we have $f(x_0)+x^{*}(x_0) \geq s$. Therefore,
for every $x \in D$ we have
$$
-f(x)-x^{*}(x) \leq -f(x)+|x^{*}(x)| \leq s-\eta+\frac{\eta}{2}<s\leq f(x_0)+x^{*}(x_0),
$$
which implies that $|f(x)+x^{*}(x)|\leq f(x_0)+x^{*}(x_0)= |(f+ x^*)(x_0)|$ for every $x\in D$.

$2^{\circ}$. Suppose that $ s = |\inf_{D}f|$.
By Theorem \ref{pre}, we can choose $x^{*}$ so that $f+x^{*}$ attains its supremum at $x_0 \in D$ and $\|x^{*}\|\leq\frac{\epsilon}{2}$.
If $|f+x^{*}|(x)\leq f(x_0)+x^{*}(x_0)$ for every $x\in D$, we are done. Otherwise, there exists $y \in D$ such that
$|f+x^{*}|(y)>f(x_0)+x^{*}(x_0)$. Clearly, $(f+x^*)(y) <0$ and
$$
-f(y)-x^{*}(y)>f(x_0)+x^{*}(x_0)\geq f(x)+x^{*}(x)
$$
for every $x\in D$. Set $g=-f-x^{*}$. Then we have $g(y)>-g(x_0)\geq -g(x)$ for every $x \in D$, which means that $\sup_{D}g > |\inf_{D} g|$.
 It follows from the case $1^{\circ}$ that there exist $y^{*} \in X^{*}$ and $y_0 \in D$ such that $\|y^{*}\|\leq\frac{\epsilon}{2}$ and $|g+y^{*}|(x) \leq (g+y^{*})(y_0)$ for every $x\in D$. Therefore,
$$
|f+ (x^{*}-y^{*})|(x)=|g+y^{*}|(x)\leq (g+y^{*})(y_0) = (-f-x^{*}+y^{*})(y_0)\leq|f+(x^{*}-y^{*})|(y_0).
$$
for every $x\in D$ and $\|x^{*}-y^{*} \|\leq \frac{1}{2}\epsilon+\frac{1}{2}\epsilon\leq\epsilon$.

$3^{\circ}$  Suppose that $\sup_{D} f < |\inf_{D}f|$. We can prove this case by applying the case $1^{\circ}$ to $(-f)$.

Further, if $D$ is symmetric, we note that $\|f\|_{D}=\sup_{D}f=|\inf_{D}f|$. From the assumption and Theorem \ref{first},
we can choose $x^{*}\in X^{*}$ and $x_1 \in D$ so that
$$
\sup_{D}(f+x^{*})=(f+x^{*})(x_1),~~~\|x^{*}\|\leq \epsilon~~~\mbox{and}~~~\|x_1-x_0\|\leq\frac{\delta}{\epsilon}.
$$
Since $(f+x^{*})(x_1)=\|f+x^{*}\|_{D}$, it means that $|f+x^{*}|$ attains its supremum at $x_1$ on $D$.
\end{proof}

If $D$ is not symmetric in the above theorem, we can hardly choose $x_1\in D$ satisfying $\| x_0-x_1\| \leq \frac{\delta}{\epsilon}$ and also
that $|f + x^*|$ attains its supremum at $x_1$ on $D$. Indeed, for $f\in S_{X^*}$ which does not attain its norm and for $0<\epsilon <\frac{1}{3}$, we let $$S_{1}=\{x\in B_X : f(x) \geq 1-\epsilon^2 \},~~ T=B_X \cap \ker f,~~\mbox{and}~~ S_{2}=\overline{(-S_{1} + T)}.$$ Note that $S_2$ is a closed bounded convex subset of $2B_X$.
We set $D=\overline{co}(S_{1}\cup S_{2})$. Clearly, $\|f\|_D =1$ and we can see easily that there exists $x_0\in D\cap S_X$
such that $f(x_0)> 1- \frac{\epsilon^2}{2}$.

We claim that for every $\|x^*\|\leq \epsilon$ the function $|f+x^*|$ cannot attain its supremum on $D$
at any point $z\in D$ with $\|z-x_0\|\leq \epsilon$. Otherwise, there exists $\|x^*\|\leq \epsilon$ such that
the function $|f+x^*|$ attains its supremum on $D$
at some point $z\in D$ with $\|z-x_0\|\leq \epsilon$. Since $$(f+x^*)(z) = (f+x^*)(x_0) + (f+x^*) (z-x_0)\geq 1- \frac{\epsilon^2}{2} -\epsilon-\epsilon (1+\epsilon) >0,$$ we can deduce that $(f+x^*)(z)=|(f+x^*)(z)|$ is $\sup_D (f+x^*)$. Choose a sequence
$\{z_n\}$ in $co(S_{1}\cup S_{2})$ converging to $z$. For each $n\in \N$ we can write $z_n = (1-\lambda_n)x_n + \lambda_n y_n$, where
$x_n\in S_1$, $y_n\in S_2$ and $0\leq \lambda_n\le 1$.  An easy computation shows that
$$(f+x^*)(z)\geq 1-2\epsilon-\frac{3\epsilon^2}{2}~~\mbox{and}~~
(f+x^*)(y_n)\leq -1 + \epsilon^2 + 2\epsilon.$$ It follows from these inequalities that for each $n\in \N$

\begin{eqnarray*}
(1+\epsilon) \|z-z_n\| &>& |(f+x^{*})(z-z_n)| =
|(f+x^{*})(z)-(f+x^{*})((1-\lambda_n)x_n + \lambda_n y_n ) |\\
&=& (1-\lambda_n)\left((f+x^{*})(z)-(f+x^{*})(x_n)\right) + \lambda_n \left((f+x^{*})(z)-(f+x^{*})(y_n)\right) \\
&\geq&  \lambda_n \left((1-2\epsilon-\frac{3\epsilon^2}{2})-(-1 + \epsilon^2 + 2\epsilon) \right)= \lambda_n (2-4\epsilon -\frac{5\epsilon^2}{2}) \geq 0.
\end{eqnarray*}

 Therefore, we have that for each $n$ $$0\leq \lambda_n \leq \frac{(1+\epsilon)\|z-z_n\|}{(2-4\epsilon -\frac{5\epsilon^2}{2})},$$ which implies that
 $\lambda_n\to 0$ as $n\to \infty$. It means that $z\in S_1$.

 Now we recall that $f+x^*$ attains its supremum on $D$, but $f$ doesn't attain its supremum on $D$ because $\sup_{B_X} f = \sup_{D} f$ and $f$ doesn't attain its norm. Hence  $x^{*} \neq \alpha f$ for any $\alpha \in \R$, and there exists $w \in T$ with $x^{*}(w)<0$ and $\|w\|<\epsilon$. Since $-z+w \in (-S_1 + T) \subset D$, we obtain
 $$\sup_D|f+x^*| \geq |(f+x^*)(-z+w)|>(f+x^{*})(z)= |(f+x^{*})(z)|,$$ which is a contradiction.

\section{Operators attaining their suprema in norm on bounded closed convex sets}
It was shown in \cite{AAGM2} that for all Banach spaces $X$ and $Y$ of finite dimension $(X,Y)$ has the \emph{BPBp} for $B_X$. We show that it is still true for arbitrary bounded
closed convex subsets which are not necessarily symmetric.

\begin{theorem}
Let $X$ and $Y$ be finite dimensional Banach spaces. Then the pair $(X,Y)$ has the \emph{BPBp} on every bounded convex closed subset $D$ of $X$.
\end{theorem}

\begin{proof}
Otherwise, there exists a bounded closed convex subset $D$ in $X$ satisfying following condition: For some $\epsilon_0$, we can find
$T_{n}\in \mathcal{L}(X,Y)$ such that for every $T \in \mathcal{L}(X,Y)$ with $\|T_{n}-T\|\leq\epsilon_0$, there is $x^{T}_{n}\in D$ satisfying
$\|T_{n}(x^{T}_{n})\|>\|T_{n}\|_{D}-\frac{1}{n}$ and dist$(x^{T}_{n},NA_{D}(T))\geq\epsilon_0$, where $NA_{D}(T)=\{z\in D : \|T(z)\|=\|T\|_{D}\}$.
By finite dimensionality, we can choose $T_{n}$ converging to $T_0 \in \mathcal{L}(X,Y)$ and $x^{T_0}_{n}$ converging to $x_0 \in D$.
Then we can easily show that $\|T_0 (x_0) \|=\|T_0\|_{D}$ and $\|x^{T_0}_{n}-x_0\| \geq \epsilon_0$, which contradicts $\|x^{T_0}_{n}-x_{0}\|\rightarrow 0$.
\end{proof}

J. Lindenstrauss \cite{Lindens} introduced the notion of property $\beta$: A Banach space $Y$ is called to have property $\beta$ if there is $0 \leq\lambda<1$ and a family $\{(y_{\alpha},f_{\alpha}) \in B_Y \times B_Y^{*} \}$ such that
(i) $f_{\alpha}(x_{\alpha})=\|x_{\alpha}\|=1$
(ii) $|f_{\alpha}(y_{\beta})| \leq \lambda$ for $\alpha \neq \beta$
(iii) $\|y\|=\sup_{\alpha}|f_{\alpha}(y)|$.

He showed that if a Banach space $Y$ has property $(\beta)$, then the set of all norm attaining operators from $X$ into $Y$ is dense in $\mathcal{L}(X,Y)$ for every Banach space $X$.
In 1982, J. Partington \cite{Partington} proved rather a surprising result that every Banach space can be renormed to have property ($\beta$). Acosta et al. \cite{AAGM2} showed that if $Y$ has property $(\beta)$, then $(X,Y)$ has the \emph{BPBp} on $B_X$ for every Banach space $X$. Now we prove that it is still true for bounded closed convex subsets of $X$.

\begin{theorem}\label{beta}
Let $Y$ be a Banach space with property $(\beta)$ and $D$ be a bounded closed convex set in $X$.
Then $(X,Y)$ has the \emph{BPp} on $D$. Moreover, if $D$ is symmetric, then $(X,Y)$ has the \emph{BPBp} on $D$.
\end{theorem}

\begin{proof}
Without loss of generality we may assume $D \subseteq B_X$, and first consider the case where $D$ is symmetric.
Let $$0<\epsilon<\frac{1-\lambda}{2+\lambda}~~\mbox{and}~~\frac{\epsilon(2+\lambda)}{1-2\epsilon-\lambda-\lambda\epsilon}\leq\eta\leq\frac{2\epsilon(2+\lambda)}{1-2\epsilon-\lambda-\lambda\epsilon}.$$
Assume that $T \in \mathcal{L}(X,Y)$, $\|T\|_{D}=1$, $\|T\|=M$ and $\|Tx_0\|>1-\frac{\epsilon^2}{2}$ for some $x_0\in D$. Then we can choose $\alpha_0$ so that $|(T^{*}f_{\alpha_0})(x_0)|=|f_{\alpha_0}(Tx_0)|>1-\frac{\epsilon^2}{2}$.
By Theorem \ref{main}, there exist
$g\in X^{*}$ and $z_0 \in D$ such that
$$
|g(z_0)|=\|g\|_{D},~~\|g-T^{*}f_{\alpha_0}\|\leq\epsilon,~~and~~\|x_0 -z_0\| \leq \epsilon.
$$
Since $\|g-T^{*}f_{\alpha_0}\|_{D}\leq\|g-T^{*}f_{\alpha_0} \|\leq\epsilon$, we have $1-2\epsilon\leq1-\frac{\epsilon^2}{2}-\epsilon \leq \|g\|_{D}
\leq 1+\epsilon.$
Define $T_0 \in \mathcal{L}(X,Y)$ by
$$
T_0 (x)=T(x)+((1+\eta)g(x)-T^{*}f_{\alpha_0}(x))y_{\alpha_0}.
$$
Clearly
$$
\|T^{*}_0 f_{\alpha_0}\|_{D}=(1+\eta)\|g\|_{D} \geq (1+\eta)(1-2\epsilon).
$$
For $\alpha \neq \alpha_0$,
\[ \begin{array}{lcl}
\|T^{*}_0 f_{\alpha}\|_{D} &\leq&
\|T^{*}f_{\alpha}\|_{D}+\lambda(\|g-T^{*}f_{\alpha_0}\|_{D}+\eta \|g\|_{D})
\\[5pt]
&\leq& 1+\lambda(\epsilon+\eta (1+\epsilon))
\\[5pt]
&\leq& (1+\eta)(1-2\epsilon),
\end{array} \]
where the last inequality follows from $\eta \geq \frac{\epsilon(\lambda+2)}{1-2\epsilon-\lambda-\lambda\epsilon}$.
Since $\|T^{*}_0 f_{\alpha}\|_{D} \leq \|T^{*}_0 f_{\alpha_0}\|_{D}$ for all $\alpha \neq \alpha_0$, we have
$$
\|T_0\|_{D}=\sup_{\alpha}\{\|T^{*}_0 f_{\alpha}\|_{D}\}=\|T^{*}_0 f_{\alpha_0}\|_{D}=f_{\alpha_0}(T_0 z_0)
\leq \|T_0 z_0\|\leq \|T_0\|_{D}.
$$
Therefore, $T_0$ attains its supremum at $z_0\in D$ with $\|x_0 -z_0 \|\leq \epsilon$ and
$$
\|T-T_0\|\leq\|g-T^{*}f_{\alpha_0}\|+\eta\|g\| \leq \epsilon+\eta (M+\epsilon)\leq \epsilon\left(1+\frac{2(2+\lambda)(M+\epsilon)}{1-2\epsilon-\lambda-\lambda\epsilon}\right).
$$

For a bounded closed convex subset $D$, by Theorem \ref{main}, we can choose $g \in X^{*}$ and $z_0 \in D$ so that
$$
\|g\|_{D}=|g(z_0)|~~and~~\|g-T^{*}f_{\alpha_0}\|\leq\epsilon.
$$
The rest of proof follows similarly.
\end{proof}

Recall that the modulus of convexity of a Banach space $X$ is defined on $B_X$ by
$$
\delta(\epsilon)=\inf \left\{ 1-\frac{\|x+y\|}{2} : x,y\in B_{X},~~\|x-y\|\geq\epsilon \right\}.
$$
We can naturally extend this notion for a bounded closed absorbing convex set $D$. We define $\delta_{D}(\epsilon)$ for $0<\epsilon <1$  by
$$
\delta_{D}(\epsilon)=\inf \left\{ \frac{1}{2}\rho_{D}(x)+\frac{1}{2}\rho_{D}(y)-\rho_{D}\left(\frac{x+y}{2}\right) : x,y\in D,~~\rho_{D}(x-y)\geq\epsilon \right\},
$$
where $\rho_{D}$ is the Minkowski functional of $D$, that is, $\rho_{D}(x)=\inf\{\lambda >0 : x \in \lambda D \}$.

 In the following we get such a general result that for a bounded closed absorbing convex subset $D$ of $X$ with positive modulus convexity, the pair $(X,Y)$ has \emph{BPBp} on $D$ for every Banach space $Y$.

\begin{theorem}
Let $X$ and $Y$ be (real or complex) Banach spaces and $D$ be a bounded closed absorbing convex subset of $B_X$ such that $\delta_{D}(\epsilon)>0$ for every $0<\epsilon <\frac{1}{2}$. If
$T \in S_{\mathcal{L}(X,Y)}$ and $x_1 \in D$ satisfy
$$
\|Tx_1\|>\|T\|_D -\epsilon^{3}\delta_{D}(\epsilon),
$$
 for sufficiently small $\epsilon$ relatively to $\|T\|_{D}$, then there exist $S \in \mathcal{L}(X,Y)$ and $z\in D$ such that $\|Sz\|=\|S\|_{D}$, $\|S-T\|<\frac{4\epsilon^2}{1-\epsilon}$ and $\|x_1 -z\|\leq \rho_D(x_1-z)<\frac{\epsilon}{1-\epsilon}.$
\end{theorem}

\begin{proof}  Let $T_1 = T$. Choose $f_1\in S_{Y^{*}}$ so that
$$
f_1 (T_1 x_1)=\|T_1 x_1\|>\|T_1\|_{D}-\epsilon^{3}\delta_{D}(\epsilon).
$$
Inductively choose $\{T_{k} \}_{k=2}^{\infty}$, $\{x_k \}_{k=2}^{\infty} \subseteq D$, and $\{f_k\}_{k=2}^{\infty} \subseteq S_{Y^{*}}$ satisfying
$$
T_{k}(x)=T_{k-1}(x)+\epsilon^{k}f_{k-1}(T_{k-1}x)T_{k-1}x_{k-1},
$$
$$
\|T_{k}x_{k}\|>\|T_{k}\|_{D}-\epsilon^{k+2}\delta_{D}(\epsilon^{k}),~~\rho_{D}(x_{k})=1,
$$
$$
f_{k-1}(T_{k-1}x_{k})=|f_{k-1}(T_{k-1}x_{k})|,
$$
$$
and~~~f_{k}(T_{k}x_{k})=\|T_{k}x_{k}\|.
$$
Since $\|T_{k}\|<2$ and $\|T_{k+1}-T_{k}\| \leq 2 \epsilon^{k+1} \|T_{k}\|$ for every $k$, $\{T_{k}\}_{k=1}^{\infty}$ converges to $S$ and $\|T-S\|\leq\frac{4\epsilon^2}{1-\epsilon}$.
An upper bound of $\|T_{k+1}\|_{D}$ is
\begin{eqnarray*}
\|T_{k+1}\|_{D} &<&
\|T_{k+1}x_{k+1}\|+\epsilon^{k+3}\delta_{D}(\epsilon^{k+1})
\\
&\leq& \|T_{k}\|_{D}+\epsilon^{k+1} |f_{k}(T_{k}x_{k+1})|\cdot \|T_{k}\|_{D}+\epsilon^{k+3}\delta_{D}(\epsilon^{k+1}).
\end{eqnarray*}

A lower bound is
\begin{eqnarray*}
\|T_{k+1}\|_{D} &\geq&
\|T_{k+1}x_{k}\|=\|T_{k}x_{k}\| \cdot |1+\epsilon^{k+1}f_{k}(T_{k}x_{k})|
\\
&>& (\|T_{k}\|_{D}-\epsilon^{k+2}\delta_{D}(\epsilon^{k}))(1+\epsilon^{k+1}(\|T_{k}\|_{D}-\epsilon^{k+2}\delta_{D}(\epsilon^{k}) ))
\\
&>& \|T_{k}\|_{D}+\epsilon^{k+1}\|T_{k}\|_{D}^{2}-2\epsilon^{2k+3}\delta_{D}(\epsilon^{k})\|T_{k}\|_{D}-\epsilon^{k+2}\delta_{D}(\epsilon^{k}).
\end{eqnarray*}
Combining these two bounds yields
$$
|f_{k}(T_{k}x_{k+1})| > \|T_{k}\|_{D}-2\epsilon^{k+2}\delta_{D}(\epsilon^{k})-\frac{\epsilon\delta_{D}(\epsilon^{k})+\epsilon^{2}\delta_{D}(\epsilon^{k+1})}{\|T_{k}\|_{D}}.
$$
Since $\delta_{D}(\epsilon^{k})\geq \delta_{D}(\epsilon^{k+1})$ and $\|Tx\|\leq \rho_{D}(x)~\|T\|_{D}$ for every $x\in D$, we have
\begin{eqnarray*}
\rho_{D}\left( \frac{x_{k+1}+x_{k}}{2}\right)\|T_{k}\|_{D} &\geq&
\left\|T_{k}\left(\frac{x_{k+1}+x_{k}}{2} \right) \right\|
\geq \frac{1}{2}Re(f_{k}T_{k}x_{k+1}+f_{k}T_{k}x_{k})
\\
&\geq& \|T_{k}\|_{D}-\frac{1}{2}\left( 2\epsilon^{k+2}\delta_{D}(\epsilon^{k})+\frac{\epsilon\delta_{D}(\epsilon^{k})+\epsilon^{2}\delta_{D}(\epsilon^{k+1})}{\|T_{k}\|_{D}}
+\epsilon^{k+2}\delta_{D}(\epsilon^{k})  \right)
\\
&\geq& \|T_{k}\|_{D}-\delta_{D}(\epsilon^{k})\left( \frac{3}{2}\epsilon^{k+2}+\frac{\epsilon+\epsilon^{2}}{2\|T_{k}\|_{D}} \right).
\end{eqnarray*}

It follows from $\|T_{k}-T\|<\frac{4\epsilon^2}{1-\epsilon}<4\epsilon$,
$$
\|T\|_{D}-4\epsilon<\|T_{k}\|_{D}<\|T\|_{D}+4\epsilon
$$
$$
\rho_{D}\left(\frac{x_{k+1}+x_{k}}{2}\right) > 1- \delta_{D}(\epsilon^{k})\left( \frac{\frac{3}{2}\epsilon^{k+2}+\frac{\epsilon+\epsilon^{2}}{2\|T\|_{D}-8\epsilon}}{\|T\|_{D}-4\epsilon} \right).
$$
Since $\epsilon$ is sufficiently small relatively to $\|T\|_{D}$, then we have
\begin{eqnarray*}
\rho_{D}\left(\frac{x_{k+1}+x_{k}}{2}\right) &>& 1- \delta_{D}(\epsilon^{k})
\\
&=&\frac{1}{2}\rho_{D}(x_{k+1})+\frac{1}{2}\rho_{D}(x_{k}) - \delta_{D}(\epsilon^{k}),
\end{eqnarray*}
 which implies that $\rho_D(x_{k+1}-x_{k})\leq\epsilon^{k}$. Hence $\{x_{k} \}_{k=1}^{\infty}$ converges in norm to $z \in D$. We can also see easily that $\|x_1 -z\| \leq \rho_D(x_1-z)\leq\frac{\epsilon}{1-\epsilon}$, $\|T-S\|\leq \frac{4\epsilon^2}{1-\epsilon}$ and $\|Sz\|=\|S\|_{D}$.

\end{proof}

\begin{corollary} [\cite{KimLee}]
Let $X$ be a uniformly convex Banach space and $Y$ be a Banach space. Then $(X,Y)$ has the \emph{BPBp} on $B_X$.
\end{corollary}

\begin{remark} \label{remark}
It is easy to notice that the \emph{BPBp} is an isometric property, but not isomorphic. On the other hand, the \emph{BPBp} still holds on the image of an isomorphism in the following sense.
Let $\Psi$ be an isomorphism from a Banach space $X$ into a Banach space $Z$. We can see that if $(X,Y)$ has the \emph{BPBp} on $B_X$, then $(\Psi(X),Y)$ has the \emph{BPBp} on $D=\Psi(B_X)$. Further, if $Y$ is injective, then $(Z,Y)$ has the \emph{BPBp} on $D$. We recall that a Banach space $Z$ is called injective if for every Banach space $X$ and for every subspace $W$ of $X$, every operator from $W$ into $Z$ can be extended to an operator from $X$ into $Z$ preserving its norm.
\end{remark}

A Banach space $X$ is called an Asplund space if the set of all points of $U$ where $f$ is Fr\'echet differentiable is dense $G_{\delta}$-subset of $U$ for
every real-valued convex continuous function $f$ defined on an open convex subset $U \subseteq X$. Equivalently every $w^{*}$-compact subset of $(X^{*},w^{*})$ is $\|\cdot\|$-fragmentable.
Here we say a subset $C$ of $(X^{*},w^{*})$ is $\|\cdot\|$-fragmentable if for every nonempty bounded subset $A \subset C$ and for every $\epsilon>0$, there is a nonempty $w^{*}$-open neighborhood $V \subset X^{*}$ such that $A \cap V$ is nonempty and has $\|\cdot\|$-diameter less than $\epsilon$ (\cite{FHHMZ}).
Recall that an operator $T\in \mathcal{L}(X,Y)$ is called by an Asplund operator if it factors through an Asplund space. That is, there are an Asplund space $Z$
and operators $T_1\in \mathcal{L}(X,Z)$ and $T_2\in \mathcal{L}(Z,Y)$ such that $T=T_{2}\circ T_{1}$. For example, every weakly compact operator is an Asplund operator since it factors through a reflexive space, so that a rank one operator is an Asplund operator. We also note that the family of Asplund operators is an operator ideal, hence the sum of two Asplund operators or the composition of an operator with an Asplund operator is again an Asplund operator.

It was shown in \cite{ACK} that the \emph{BPBp} on $B_X$ holds for an Asplund operator from $X$ into $C_{0}(L)$. We can extend this result to a symmetric bounded closed convex subset $D \subset B_X$.
Some modifications of \cite[Lemma 2.3 and Theorem 2.4]{ACK} are just needed, but we give the details for the sake of completeness.
\begin{lemma} \label{asp1}
Let $D$ be a symmetric bounded closed convex subset of $B_X$ and $T\in \mathcal{L}(X,Y)$ be an Asplund operator with $\|T\|_{D}=1$ and $\|T\| \geq M\geq 1$.
If
$$
\|T(x_0)\|>1-\frac{\epsilon^2}{4M},
$$
for some $x_0 \in D$,
then for every norming set $B \subseteq B_{Y^{*}}$ and $0<\epsilon\leq\frac{M}{2}$, there exist $x^{*}\in X^{*}$, $u_0\in D$ and a $w^{*}$-open neighborhood $U$ in $X^*$ such that
$$
|x^{*}(u_0)|=1=\|x^{*}\|_{D},~~~\|x_0-u_0\|<\epsilon~~~and~~~\|z^{*}-x^{*}\|<4\epsilon,
$$
for every $z^{*}\in U\cap T^{*}(B)$.
\end{lemma}

\begin{proof}
Since $B$ is a norming set, we can choose $y^{*}_{0}\in B$ such that
$$
|y^{*}(Tx_0)|=|T^{*}y^{*}_{0}(x_0)|>1-\frac{\epsilon^2}{4M}.
$$
Define a $w^{*}$-open neighborhood in $X^{*}$ by
$$
U_1=\left\{z^{*}\in X^{*}:|z^{*}(x_0)|>1-\frac{\epsilon^2}{4M}\right\}.
$$
Since $T^{*}y^{*}_0\in T^{*}(B)\cap U_1$, we get $T^{*}(B)\cap U_1\neq\emptyset$.
Since $T^{*}(B)$ is $\|\cdot\|$-fragmentable, (\cite{ACK}), we can find a $w^{*}$-open neighborhood $U_2 \subset X^{*}$ such that
$$
U \cap T^{*}(B) \neq \emptyset ~~~and~~~diam(U\cap T^{*}(B))<\epsilon,
$$
where $U=U_1\cap U_2$.
Now fix $z^{*}_0\in U\cap T^{*}(B)$. Write $z^{*}_0=T^{*}(w^{*}_0)$ for some $w^{*}_0\in B\subset B_{Y^{*}}$.
We can see $\|z_{0}^{*}\| \leq M$ and
$$
\|z^{*}_0\|_{D}=\sup_{x\in D}|T^{*}w^{*}_0(x)|=\sup_{x\in D}|w^{*}_0(Tx)|\leq \|T\|_{D}\leq1,
$$
which implies that $|z^{*}_0(x_0)|>1-\frac{\epsilon^2}{4M}\geq \|z^{*}_0\|_{D}-\frac{\epsilon^2}{4M}$.
By Theorem \ref{main}, we can choose $x^{*}\in X^{*}$ and $u_0\in D$ such that
$$
\|x^{*}-z^{*}_0\|<\frac{\epsilon}{2M},~~~\|x_0-u_0\|<\epsilon~~~and~~~|x^{*}(u_0)|=\|x^{*}\|_{D}.
$$
It also follows from an easy computation that
$$
1-\frac{\epsilon}{M}\leq \|x^{*}\|_{D}\leq 1+\frac{\epsilon}{2M}.
$$
Let $k=\frac{1}{\|x^*\|_D}$. Clearly,
$$
\frac{2M}{2M+\epsilon}\leq k \leq \frac{M}{M-\epsilon}
$$
and
\[ \begin{array}{lcl}
\|kx^{*}-z^{*}_0\| &\leq& \|kx^{*}-x^{*}\|+\|x^{*}-z^{*}_0\|\leq (k-1)\|x^{*}\|+\frac{\epsilon}{2M}\\[5pt]
&\leq& \frac{\epsilon}{M-\epsilon}\left(M+\frac{\epsilon}{2M}\right)+\frac{\epsilon}{2M}\leq 2\epsilon+\frac{\epsilon}{2}+\frac{\epsilon}{2}\leq 3\epsilon,
\end{array} \]
Since $|kx^{*}(u_0)|=1$ from the choice of $k$,
$kx^{*}$ satisfies the desired properties. Moreover,
$$
\|kx^{*}-z^{*}\|\leq 3\epsilon+\epsilon=4\epsilon,
$$
for every $z^{*} \in U\cap T^{*}(B)$.
\end{proof}

\begin{theorem} \label{asp}
For a symmetric bounded closed convex subset $D\subseteq B_X$ and a locally compact Hausdorff space $L$,
let $T:X\rightarrow C_{0}(L)$ be an Asplund operator with $\|T\|_{D}=1$ and $\|T\|_D =M \geq 1$. Given $0<\epsilon<\frac{M}{2}$, if $x_0\in D$ satisfies that
$$
\|Tx_0\|>1-\frac{\epsilon^2}{4M},
$$
then there exist an Asplund operator $S:X\rightarrow C_{0}(L)$ and a point $u_0\in D$ such that
$$
\|S\|_{D}=\|Su_0\|=1,~~~\|x_0-u_0\|<\epsilon~~~and~~~\|T-S\|<4\epsilon.
$$
\end{theorem}

\begin{proof}
Define $\delta:L \rightarrow C_{0}(L)^{*}$ by $\delta(s)(f)=f(s)$ for $f\in C_{0}(L)$ and $s\in L$. It is easy to check that $\phi =T^{*}\circ\delta : L \rightarrow X^{*}$ is $w^{*}$-continuous.
Since $\{\delta(s):s\in L\}$ is a norming set in $B_{C_{0}^{*}(L)}$, we can find a $w^{*}$-open neighborhood $U$ and $x^{*}\in X^{*}$ by Lemma \ref{asp1}. Here we have $T(x)(s)=\phi(s)(x)$ and $T^{*}(B)=\phi(L)$.
Since $U\cap T^{*}(B)\neq\emptyset$, there is $s_0\in L$ such that $\phi(s_0)\in U$. Consider the set
$$
W=\{s\in L:\phi(s)\in U \},
$$
which is an open neighborhood of $s_0$ due to the $w^{*}$-continuity of $\phi$.
By Urysohn's lemma there exists a continuous function $f:L\rightarrow [0,1]$ satisfying
$$
f(s_0)=1~~~and~~~supp(f)\subset W.
$$
Define a linear operator $S:X\rightarrow C_{0}(L)$ by
$$
S(x)(s)=f(s)x^{*}(x)+(1+f(s))T(x)(s).
$$
Define $S_2\in \mathcal{L}(C_{0}(L),C_{0}(L))$ by $S_2(h)=(1+f)h$. Then $S(x)=f\cdot x^{*}(x)+S_2 (T(x))$,
hence $S$ is an Asplund operator.
It follows easily that $\|S\|_{D}\leq1$ and $|S(u_0)(s_0)|=|y^{*}(u_0)|=1$, which shows that $S$ attains its supremum at $u_0$ on $D$ and $\|u_0-x_0\|<\epsilon.$
For an upper bound of $\|T-S\|$,
\[ \begin{array}{lcl}
\|T-S\| &=& \sup_{x\in B_X}\|Tx-Sx\|=\sup_{x\in B_X}\sup_{s\in L}|f(s)|\cdot|T(x)(s)-x^{*}(x)|\\[5pt]
&=& \sup_{s\in W}\sup_{x\in D}|\phi(s)(x)-x^{*}(x)|\leq\sup_{s\in W}\|\phi(s)-x^{*}\|\leq 4\epsilon,
\end{array} \]
where the last inequality is derived from $\phi(s)\in U\cap T^{*}(B)$ and Lemma \ref{asp1}. This completes the proof.
\end{proof}

\begin{corollary}
For any (real) Asplund space $X$ and any locally compact Hausdorff space $L$, the pair $(X,C_0 (L))$ has the \emph{BPBp} on a symmetric bounded closed convex set $D$ of $X$.
\end{corollary}

\begin{remark}
By the remark after Theorem \ref{main}, without the symmetry of $D$ we can show only that there exists an Asplund operator $X \rightarrow C_{0}(L)$ such that
$$
u_0 \in D,~~~\|S\|_{D}=1=\|S(u_0)\|~~~and~~~\|T-S\|<4\epsilon,
$$
by modifying the proof of Lemma \ref{asp1} and by the first part of Theorem \ref{main}.
\end{remark}

\section{Stability of the Bishop-Phelps-Bollob\'as property on direct sums}

In order to compare the function $\eta(\epsilon)$ appearing in the definition of the \emph{BPBp} for different pairs $(X,Y)$, the notion of $\eta(X,Y)(\epsilon)$ was introduced in \cite{ACKLM}. We now generalize it to a bounded closed convex subset $D$ of $B_X$.

\begin{definition}
Let $X$ and $Y$ be (real or complex) Banach spaces. For a bounded closed convex subset $D \subseteq B_X$ and $T \in \mathcal{L}(X,Y)$,
$$
\Pi_{D}(X,Y)=\{(x,T):x\in D, \|T(x)\|=\|T\|_{D}=1 \}
$$
$$
\eta_{D}(X,Y)(\epsilon)=\inf\{1-\|Tx\|:x\in D, \|T\|_{D}=1, \dist((x,T),\Pi_{D}(X,Y))\geq\epsilon \},
$$
where $\dist((x,T),\Pi_{D}(X,Y))=\inf\{\max\{\|x-y\|,\|T-S\|\}: (y,S)\in \Pi_{D}(X,Y) \}$.
\end{definition}
It is clear the pair $(X,Y)$ has the \emph{BPBp} on $D$ if and only if $\eta_{D}(X,Y)(\epsilon)>0$ for every $0< \epsilon <1$. If a function $\epsilon \longmapsto \eta_{D}(\epsilon)$ is valid in the definition of the \emph{BPBp} on $D$ for the pair $(X,Y)$, then $\eta_{D}(\epsilon) \leq \eta_{D}(X,Y)(\epsilon)$. In other words, $\eta_{D}(X,Y)(\epsilon)$ is the largest function we can find to ensure that $(X,Y)$ has the \emph{BPBp} on $D$. It is clear that $\eta_{D}(X,Y)(\epsilon)$ is increasing with respect to $\epsilon$.

Let $\{X_i : i\in I\}$ and $\{Y_j :j \in J\}$ be families of Banach spaces, $X=(\bigoplus_{i\in I}X_{i})_{\ell_{1}}$ and
$Y=(\bigoplus_{j\in J}Y_{j})_{\ell_{\infty}}$.
Let $E_{i}$ and $F_{j}$ be the natural isometric embeddings of $X_{i}$ and $Y_{j}$ into $X$ and $Y$, respectively and let $P_{i}$ and $Q_{j}$ be the canonical
projections of norm one from $X$ and $Y$ onto $X_{i}$ and $Y_{j}$, respectively. For $D\subset X$ we let $D_{i}=\overline{P_{i}(D)}^{X_{i}}$ for each $i\in I$.

Pay\'a and Saleh \cite{PaySal} studied the denseness of norm attaining operators from the $\ell_1$-sum of domain space into the $\ell_{\infty}$-sum of range spaces. Their methods in \cite{PaySal} were applied in studying the Bishop-Phelps-Bollob\'as property for operators on those spaces
(\cite{ACKLM, ChoiKim}). With some suitable condition on $D$, we have the following analogous results to \cite[Theorem 2.1]{ACKLM}.

\begin{proposition}
Let $X=(\bigoplus_{i\in I}X_{i})_{\ell_{1}}$,
$Y=(\bigoplus_{j\in J}Y_{j})_{\ell_{\infty}}$ and $D$ be a bounded closed convex subset of $B_X$. Suppose that $D=\overline{co} (\cup E_i D_i)\subset B_{X}$. If the pair $(X,Y)$ has the \emph{BPBp} with $\eta_{D}(\epsilon)$ on $D$, then the pair $(X_{i},Y_{j})$ has the \emph{BPBp} on $D_{i}$ for every $i\in I$ and for every $j\in J$.
More precisely,
$$
\eta_{D}(X,Y)(\epsilon) \leq \eta_{D_{i}}(X_{i},Y_{j})(\epsilon),~~(i\in I,~j\in J).
$$
\end{proposition}

\begin{proof}
Fix $h\in I$ and $k\in J$. Suppose that $\|T(x_h)\|>1-\eta_{D}(\epsilon)$ for $T \in \mathcal{L}(X_{h},Y_{k})$, $\|T\|_{D_h}=1$ and $x_h\in D_{h}$.
Define an operator $\widetilde{T}=F_{k}TP_{h}\in \mathcal{L}(X,Y)$. It is easy to check that
$$
\|\widetilde{T}\|_{D}=\|T\|_{D_h},
$$
so that $\|\widetilde{T}(E_{h}x_{h})\|>1-\eta_{D}(\epsilon)$. The assumption gives us $(u ,\widetilde{S})\in\Pi_{D}(X,Y)$ such that
$$
\|\widetilde{S}-\widetilde{T}\|<\epsilon~~~and~~~\|u-E_{h}x_{h}\|<\epsilon.
$$
Define $S=Q_{k}\widetilde{S}E_{h}\in\mathcal{L}(X_{h},Y_{k})$. Then $\|S-T\|\leq\|\widetilde{S}-\widetilde{T}\|<\epsilon.$
For $j\neq k$, we have $Q_{j}\widetilde{T}=0$ by the definition of $\widetilde{T}$, which implies that
$$
\|Q_{j}\widetilde{S}\|_{D}=\|Q_{j}\widetilde{S}-Q_{j}\widetilde{T}\|_{D}\leq \|\widetilde{S}-\widetilde{T}\|\leq\epsilon.
$$
Since the range of $\widetilde{S}$ is the $\ell_{\infty}$-sum of $Y_{j}$'s, we have $\|Q_{k}\widetilde{S}\|_{D}=\|\widetilde{S}\|_{D}=1=\|\widetilde{S}(u)\|=
\|Q_{k}\widetilde{S}u\|$.
It follows from the assumption $D=\overline{co} (\cup E_i D_i)$ that every $u \in D$ can be written by
$$
u=\sum_{i=1}^{\infty} \lambda_{i} E_{i} u_{i},
$$
where $u_i \in D_i$ and $\sum_{i=1}^{\infty} \lambda_{i} \leq 1$.
Indeed, choose $v^n \in co(\cup E_i D_i )$ converging to $u$. We can write $v^n =\sum_{i=1}^{\infty} \lambda_{i}^{n}E_{i}v_{i}^{n}$ ($v_{i}^{n}\in D_{i}$), where $v_{i}^{n}=0$ and $\lambda_{i}^{n} =0$ except finitely many $i$'s. Since $E_i D_i \cap E_j D_j = \{0\}$ for $i \neq j$, we have that
$\lambda_{i}^{n} v_{i}^{n} \to P_i u $ for every $i$ as $n\to \infty$. By the diagonal argument, up to a subsequence, there exists a sequence $\{\lambda_{i}\}$ such that $\lambda_{i}^{n} \rightarrow \lambda_{i}\geq 0$ for every $i$ as $n \to \infty$. By the Fatou lemma we obtain that $\sum_{i=1}^{\infty} \lambda_{i=1} \leq 1$. We can also see that there exists $u_i \in D_i$ for every $i$ such that $v_{i}^{n} \rightarrow u_{i}$ as $n \to \infty$ and $\lambda_{i} u_{i}=P_{i}u$, hence $u = \sum_{i=1}^{\infty} \lambda_{i}E_{i}u_{i}$.

Since $\| Q_k\tilde{S}(u) \| =1$ and $E_i u_i \in D$ for every $i$, we have that $\|Q_k\tilde{S}(E_i u_i) \| =1$ for every $i$ where $\lambda_{i} \neq 0$ and also that $\sum_{i=1}^{\infty}\lambda_i =1$.
It follows from $\tilde TE_i =0$ for $i\neq h$ that
$\|Q_k\tilde{S}E_i\| \leq \|\tilde{T} - \tilde{S}\| \leq \epsilon$ for $i\neq h$.
Therefore,
$$1=\|Q_k\tilde{S}u\| \leq \lambda_h \|Q_k\tilde{S}E_h u_h\| + \epsilon \sum_{i\neq h} \lambda_i \leq \sum \lambda_i =1,$$
which implies that $\lambda_i =0$ for $i\neq h$, $\lambda_h=1$ and $\|S(u_h)\| = \|Q_k\tilde{S}E_h u_h\|=1=\|S\|_{D_h}$.
Further,
$$
\|u_h -x_h\|=\|P_h (u-E_h x_h)\| \leq \|u - E_h x_h\| \leq \epsilon.
$$
\end{proof}

If we fix the domain space $X$, then the reverse inequality also holds for the $\ell_{\infty}$-sum of range spaces.

\begin{proposition}\label{ell-infty}
$\eta_{D}(X,Y)=\inf_{j\in J}\eta_{D}(X,Y_{j})$
\end{proposition}

\begin{proof}
It is enough to prove that $\eta_{D}(X,Y)\geq\inf_{j\in J}\eta_{D}(X,Y_{j})$. Fix $\epsilon\in (0,1)$.
Let $0\leq \alpha=\inf_{j\in J}\eta_{D}(X,Y_{j})(\epsilon)<1$.
For $0<\alpha <1$, suppose that $\|Tx_0\|>1-\alpha$ for $x_0\in D$
and $T\in \mathcal{L}(X,Y)$ with $\|T\|_{D}=1$. We can choose $k\in J$ so that $\|Q_{k}Tx_0\|>1-\alpha$.
Then there exist $S_{k}:X\rightarrow Y_{k}$ and $u\in D$ such that
$$
\|S_{k}u\|=\|S_{k}\|_{D}=1,~~~\|S_{k}-Q_{k}T\|<\epsilon~~~and~~~\|x_0 -u\|<\epsilon.
$$
Define $S:X\rightarrow Y$ by
$$
S=\sum_{j\neq k}F_{j}Q_{j}T + F_{k}S_{k}.
$$
It is easy to check that $(u,S)\in\Pi_{D}(X,Y)$. Moreover
$$
\|T-S\|=\sup_{j\in J}\|Q_{j}(T-S)\|=\|Q_{k}T-S_{k}\|<\epsilon,
$$
which means that $\eta_{D}(X,Y)(\epsilon)\geq\alpha.$
\end{proof}
 We now consider the case where $X$ is the $\ell_{\infty}$-sum of domain spaces $X_i$.

\begin{proposition}
Let $X=[\bigoplus_{i\in I}X_{i}]_{\ell_{\infty}}$. Assume that $D$ is a bounded closed convex subset of $B_X$,
$D=\Pi_{i\in I}D_i$ and that there exists $\epsilon_0 > 0$ such that $\frac{\lambda x_{i}}{\|x_i\|} \in D_i$ for every $x_i \in D_i$ and for every $0 \leq \lambda \leq \epsilon_0$.
If the pair $(X,Y)$ has the \emph{BPBp} on $D$ with $\eta(\epsilon)$, then the pair $(X_{i},Y)$ has the \emph{BPBp} on $D_{i}$ with $\eta(\epsilon)$ for every $i\in I$. More precisely,
$$
\eta_{D_i}(X_i, Y)(\epsilon) \geq\eta_{D}(X,Y)(\epsilon)~~~\mbox{for every}~i\in I.
$$
\end{proposition}

\begin{proof}
Fix $h\in I$. Suppose that $\|T(x_h)\|>1-\eta(\epsilon)$ for some $T\in \mathcal{L}(X_h,Y)$ with $\|T\|_{D_h}=1$ and $x_h \in P_h(D)$. Define $\tilde{T} \in \mathcal{L}(X,Y)$ by
$\widetilde{T}(u_h,z)=T(u_h)$, where $(u_h,z)\in P_{h}X\oplus (I-P_{h})X$. Then $\|\widetilde{T}\|_{D}=1$ and $\|\widetilde{T}(E_{h}x_h)\|>1-\eta(\epsilon)$.
Since $E_{h}x_{h}\in D$ and the pair $(X,Y)$ has the \emph{BPBp} on $D$, for $0<\epsilon<\epsilon_0$ there exist $\widetilde{S}\in\mathcal{L}(X,Y)$ with $\|\widetilde{S}\|_{D}=1$ and $u \in D$ such that
$$
\|\widetilde{S}(u)\|=1,~~~\|\widetilde{S}-\widetilde{T}\|<\epsilon,~~~and~~~\|E_{h}x_h-u\|<\epsilon.
$$
Now we define an operator $S\in\mathcal{L}(X_h,Y)$ for $u_h\in X_{h}$
$$
S(u_h)=\widetilde{S}(E_{h}u_h).
$$
From $\|E_{h}x_h-u\|<\epsilon$, we get $\|P_{i}u\|<\epsilon$ for $i\neq h$. The assumption yields that $\frac{\epsilon_0}{\epsilon}P_{i}(u)\in D_i$ for $i\neq h$. Let $w$ be the element in $D$ such that $P_{h}(w)=P_{h}(u)$ and $P_{i}(w)=\frac{\epsilon_0}{\epsilon}P_{i}u$ for $i\neq h$. Then
$$
\widetilde{S}(w)=\widetilde{S}(E_{h}P_{h}u)+\sum_{i\neq h}\frac{\epsilon_0}{\epsilon}\widetilde{S}(E_{i}P_{i}u),
$$
hence,
$$
\widetilde{S}(u)=\left(1-\frac{\epsilon}{\epsilon_0}\right)\widetilde{S}(E_{h}P_{h}u)+\frac{\epsilon}{\epsilon_0}\widetilde{S}(w).
$$
It follows from $\|\widetilde{S}\|_{D}=1=\|\tilde{S}(u)\|$ that $\|\widetilde{S}(E_{h}P_{h}u)\|=\|\widetilde{S}(w)\|=1$.
Since $\|S\|_{D_h}\leq \|\widetilde{S}\|_{D}=1$, $S$ attains its maximum at $P_{h}u $ on $D_{h}$. Moreover,
$$
\|S-T\|<\epsilon~~~and~~~\|x_h-P_{h}u\|=\|E_{h}x_h -E_{h}P_{h}u\| \leq \|E_h x_h -u\|<\epsilon,
$$
so that the proof is completed.
\end{proof}

Examples satisfying the above assumption on $D$ include $\bigoplus_{i\in I} \lambda_{i}B_{X_{i}}$ with $\inf_{i\in I}\lambda_{i}>0$
as well as $B_X$. The case of the $\ell_{1}$-sum of range spaces follows immediately from \cite[Proposition 2.7]{ACKLM}, so we omit the proof.

\begin{proposition}
Let $D$ be a bounded closed convex subset of $B_X$ and $Y=[\bigoplus_{j\in J}Y_j ]_{\ell_1}$. If the pair $(X,Y)$ has the the \emph{BPBp} on $D$ with $\eta(\epsilon)$, then the pair $(X,Y_j)$ also has the the \emph{BPBp}on $D$ with $\eta(\epsilon)$ for every $j\in J$. More precisely, for every $j\in J$,
$$
\eta_D(X,Y) \leq \eta_D(X,Y_j)
$$
\end{proposition}

A Banach space $X$ is called a universal \emph{BPB} domain space if for every Banach space $Z$, the pair $(X,Z)$ has the \emph{BPB}p on $B_X$. It was proved in \cite{ACKLM} that the base field $\R$ or $\C$ is the unique Banach space which is a universal \emph{BPB} domain space in any equivalent renorming. Its proof follows immediately from \cite[Lemma 3.2]{ACKLM}: Let $X$ be a Banach space containing a non-trivial $L$-summand and $Y$ be a strictly convex Banach space. If the pair $(X,Y)$ has the \emph{BPB}p on $B_X$, then $Y$ is uniformly convex. We can extend this result to a bounded closed convex subset $D$ of $X$. With proposition \ref{univ} and remark \ref{remark}, using similar argument in \cite{ACKLM}, we can say that the the
base field $\R$ or $\C$ is the unique \emph{BPB} domain on every bounded closed convex subset $D$.

\begin{proposition}\label{univ}
Let $X$ be a (real) Banach space containing a nontrivial $L$-summand, i.e. $X=X_{1}\oplus_{1}X_{2}$ for some non trivial subspaces $X_1$ and $X_2$. Let $D$ be a bounded closed convex subset of $B_X$ such that $D=\overline{co}(E_{1}D_{1}\cup E_{2}D_{2})$.
If $Y$ is a strictly convex space and if the pair $(X,Y)$ has the \emph{BPBp} on $D$, then $Y$ is a uniformly convex space.
\end{proposition}

\begin{proof}
To prove that $Y$ is uniformly convex, for every $0<\epsilon<1/2$ we need to find $\delta(\epsilon)>0$ such that $\|y_1\|=1=\|y_2\|$ and $\|\frac{y_1 + y_2}{2}\|>1-\delta(\epsilon)$ implies that $\|y_1-y_2\|<\epsilon.$
By the Bishop-Phelps theorem, we can choose $e_{1}^{*}\in X_{1}^{*}$ such that
$$
\|e_{1}^{*}\|_{D_1}=e_{1}^{*}(e_1)=1,
$$
for some $e_1 \in D_{1}$.
We can extend (still call it by) $e_{1}^{*}\in X^{*}$ by $e_{1}^{*}(x_2)=0$ for every $x_2\in X_{2}$. Similarly we can choose $e_{2}^{*}\in X^{*}$ and $e_2 \in D_2$ as a counterpart of $e_{1}^{*}$ and $e_1$, respectively. Let $M \geq 4$ be a constant with $\|e_{1}^{*}\|<M$ and $\|e_{2}^{*}\|<M$.
Suppose that $y_1 , y_2 \in S_{Y}$ and $\|\frac{y_1 + y_2}{2}\|>1-\eta_{D}(\frac{\epsilon}{2M})$. Define an operator $T\in \mathcal{L}(X,Y)$ by
$$
T(u_1,u_2)=e_{1}^{*}(u_1)y_1 + e_{2}^{*}(u_2)y_2.
$$
Then $\|T\|<2M$ and $\|T(\frac{1}{2}E_1(e_1)+\frac{1}{2}E_2(e_2))\|=\|\frac{y_1 + y_2}{2}\|>1-\eta_{D}(\frac{\epsilon}{2M})$.
Since $\frac{1}{2}E_1(\e_1)+\frac{1}{2}E_2(e_2) \in D$ and since $(X,Y)$ has the \emph{BPBp} on $D$, there exist $S\in \mathcal{L}(X,Y)$ and $(x_1,x_2)\in D$ such that
$$
\|S-T\|<\frac{\epsilon}{2M},~~~\left|\left|\frac{1}{2}e_1-x_1\right|\right|+\left|\left|\frac{1}{2}e_2-x_2\right|\right|<\frac{\epsilon}{2M}~~~
and~~~\|S(x_1,x_2)\|=1.
$$
We can see that $x_1 \neq 0 \neq x_2$. Indeed, if $x_1=0$, then
$$1=\|S(0,x_2)\|\leq \|T(0,x_2)\|+\frac{\epsilon}{2M} \leq |e^{*}_{2}(x_2)|+\frac{\epsilon}{M} \leq \frac{1}{2} + \frac{2\epsilon}{M} < 1,
$$
which is a contradiction.

Choose $\{z_n\}\subset co (E_{1}D_{1}\cup E_{2}D_{2})$ converging to $(x_1,x_2)$. Since $E_{1}D_{1}$ and $E_{2}D_{2}$ are convex sets, $z_n$ can be written as $\lambda_n u_n + (1-\lambda_n)v_n$, where
$u_n \in E_{1}D_{1}$ and $v_n\in E_{2}D_{2}$. Passing to a subsequence we may assume $\lambda_n \to \lambda$ as $n \to \infty$. Then it is easy to check that $\lambda u_n \to E_1 x_1$ and $(1-\lambda)v_n \to E_2 x_2$ as $n \to \infty$, which implies $0<\lambda<1$ because $x_1 \neq 0 \neq x_2$. Since $E_1 D_1$ and $E_2 D_2$ are closed, we can see that $\frac{E_1 x_1}{\lambda} \in E_1 D_1$ and $\frac{E_2 x_2}{1-\lambda} \in E_2 D_2$.

We claim that $|\lambda-\frac{1}{2}|\leq\frac{1}{2}\epsilon$.
Suppose $\lambda < \frac{1}{2}-\frac{1}{2}\epsilon$. Let $\delta=\frac{1}{2}-\lambda > \frac{\epsilon}{2}$ and $x_0=\frac{1}{2}e_1-x_1$. Then
$$
\frac{x_1}{\lambda}=\frac{e_1}{1-2\delta}+\frac{2x_0}{1-2\delta},
$$
Since $|e_{1}^{*}(x_0)|\leq \|e_{1}^{*}\|\cdot \|x_0\|<M\cdot \frac{\epsilon}{2M}=\frac{\epsilon}{2}$,
we have
$$
e_{1}^{*}\left(\frac{x_1}{\lambda}\right)=\frac{1}{1-2\delta}e_{1}^{*}(e_1)+\frac{2}{1-2\delta}e_{1}^{*}(x_0)>\frac{1}{1-2\delta}-\frac{2}{1-2\delta}\cdot\frac{\epsilon}{2}
=\frac{1-\epsilon}{1-2\delta}>1.
$$
Since $\frac{x_1}{\lambda}\in D_1$, this contradicts to $\|e_1^{*}\|_{D_1}=1$. We can use the same argument for the case $\lambda>\frac{1}{2}+\epsilon$, which completes the claim. Then
$$
\left|\left|e_1-\frac{x_1}{\lambda}\right|\right|\leq 2\left|\left|\frac{1}{2}e_1-x_1\right|\right|+\left|\left|2x_1-\frac{x_1}{\lambda}\right|\right|\leq \frac{\epsilon}{M}+\|x_1\|\cdot\left|\frac{2\lambda-1}{\lambda}\right| \leq \frac{\epsilon}{M}+4\epsilon.
$$
Since
$$
1=\|S(x_1,x_2)\|=\left|\left| \lambda S\left(\frac{x_1}{\lambda},0\right)+(1-\lambda)S\left(0,\frac{x_2}{1-\lambda}\right) \right|\right|,
$$
and since $\|S\|_{D}=1$, the inequalities $\left|\left| S\left(\frac{x_1}{\lambda},0\right)\right|\right| \leq 1$ and
$\left|\left| S\left(0,\frac{x_2}{1-\lambda}\right)\right|\right|\leq 1$ combined with the strict convexity of $Y$ yields that
$S\left(\frac{x_1}{\lambda},0\right)=S\left(0,\frac{x_2}{1-\lambda}\right)$.
Moreover
\[ \begin{array}{lcl}
\|y_1-y_2\|&=&\|T(e_1,0)-T(0,e_2)\|
\\[5pt]
&\leq& \|T(e_1,0)-S(e_1,0)\|+\left|\left| S(e_1,0)-S\left(\frac{x_1}{\lambda},0\right)\right|\right| \\[5pt]
&& + \left|\left|S\left(0,\frac{x_2}{1-\lambda}\right)-S(0,e_2)\right|\right|+\|S(0,e_2)-T(0,e_2)\|
\\[5pt]
&\leq& 2\|T-S\|+\|S\|\left(\left|\left| e_1-\frac{x_1}{\lambda} \right|\right| + \left|\left| e_2-\frac{x_2}{1-\lambda}\right|\right|\right)
\\[5pt]
&\leq& \epsilon\left(\frac{1}{M} + \left(2M+\frac{\epsilon}{2M}\right)\cdot 2(\frac{1}{M}+4) \right),
\end{array} \]
where the last inequality follows from $\|S-T\|<\frac{\epsilon}{2M}$ and $\|T\|<2M$. Therefore $Y$ is uniformly convex.
\end{proof}

\end{document}